\documentclass[reqno]{amsart}
\usepackage[left=1.02in,top=1.0in,right=1.02in,bottom=1.0in]{geometry}
\usepackage{tikz}
\usetikzlibrary{arrows.meta}
\usepackage{mathrsfs}
\usepackage{amsmath}
\usepackage{amssymb}
\usepackage{mathtools}
\usepackage{amsthm}
\usepackage{setspace}
\usepackage{bm}
\usepackage[all]{xy}
\usepackage{enumerate}
\usepackage{enumitem}
\usepackage{indentfirst}

\usepackage{comment}
\usepackage{color}

\allowdisplaybreaks
\numberwithin{equation}{section}

\makeatletter
\renewcommand{\subsection}{%  
  \@startsection{subsection}% 
  {2}%
  {0pt}%
  {.8\baselineskip}%
  {.4\baselineskip}%
  {\normalfont\bfseries}%
}
\makeatother

\newcommand{\cF}{\mathcal{F}}
\newcommand{\cO}{\mathcal{O}}
\newcommand{\QQ}{\mathbb{Q}}
\newcommand{\fp}{\mathfrak{p}}
\newcommand{\fm}{\mathfrak{m}}
\newcommand{\pr}{\partial}
\newcommand{\dd}[1]{\frac{\pr}{\pr #1}}
\newcommand{\Der}[1]{\mathrm{Der}_k(#1)}
\newcommand{\Frac}{\mathrm{Frac}}
\newcommand{\Fh}{\phi^{(3)}}
\newcommand{\Fs}{\psi^{(3)}}
\newcommand{\epE}{\epsilon_{\cF}(E)}
\newcommand{\divides}{\,|\,}

\title{Quotients by $(p-1)/p$-klt Foliations on Surfaces}
\author{Yutaro Hiroi}
\date{}
\address{Department of Mathematics\\ Graduate School of Sciences\\ the University of Osaka, Osaka, Japan}
\email{u507936d@ecs.osaka-u.ac.jp}

\theoremstyle{plain}
\newtheorem{dfn}{Definition}[subsection]
\newtheorem{prop}[dfn]{Proposition}
\newtheorem{lem}[dfn]{Lemma}
\newtheorem{cor}[dfn]{Corollary}
\newtheorem{thm}[dfn]{Theorem}

\theoremstyle{definition}
\newtheorem{ex}[dfn]{Example}
\newtheorem{rem}[dfn]{Remark}

\begin{document}

\begin{abstract}
We study the relation between birational singularities of 1-foliations and those of their quotients. We prove that the quotient $X/\cF$ is log canonical (resp. klt) if and only if $\cF$ is $\frac{p-1}{p}$-log canonical (resp. $\frac{p-1}{p}$-klt). Moreover, we obtain the classification of klt quotients by 1-foliations on regular surfaces in the cases $p=2,3$ and $5$.
\end{abstract}

\maketitle

\section{Introduction}
\stepcounter{subsection}
On a normal variety $X$ over an algebraically closed field $k$ of characteristic $p>0$, a 1-foliation is a saturated subsheaf of the tangent sheaf $T_X$ that is closed under Lie brackets and $p$-th powers. The foliation $\cF$ on an affine open subset $U=\text{Spec}(A)$ of $X$ is a finitely generated submodule of $\Der{A}$, which determines the subring of constants $A^{\cF}$. By gluing affine schemes $\text{Spec}(A^{\cF})$, we can define the quotient $X/\cF$. 

An interesting aspect of foliations is that of birational singularities. We can define a few classes of singularities in a similar way as we do for singularities of varieties that are considered in birational geometry. Specifically, we define the discrepancy $a(E;\cF)$ and the notions of klt and log canonical (lc) singularities for foliations. 
In recent years, birational geometry of foliations in characteristic $0$ is actively studied in algebraic geometry.

It is natural to ask how the singularities of a foliation $\cF$ on $X$ and those of the quotient $X/\cF$ are related. There are several known results on this question obtained in \cite{PosvaA}. In particular, Posva proved the following theorem on the quotients of regular surfaces. 
\begin{thm}[Theorem \ref{Posva}]
Let $S$ be a regular surface over $k$, and $\cF$ a 1-foliation of rank 1 on $S$. Then $S/\cF$ is F-regular if and only if $\cF$ is log canonical.
\end{thm}
On the other hand, it is known that every singularity of the quotient by an lc foliation on a regular surface is a toric singularity of type $\frac{1}{p}(1,\lambda)$
(Proposition \ref{multi}). In addition, if $p \geq 7$, $F$-regularity is equivalent to being klt in dimension $2$. These results complete the classification of klt quotients by $1$-foliations on regular surfaces in case $p \geq 7$. 

However, in the cases $p=2,3$ and $5$, there are some klt quotients that are not F-regular. For example, in $p=2$, the quotient of $k[x,y]$ by the derivation $(x^2+mxy^{m-1}) \partial_x + y^m \partial_y\;(m \geq 2)$ has a rational double point of type $D_{2m}^0$, and in $p=3$, the quotient of $k[x,y]$ by the derivation $y\partial_x+x^3\partial_y$ has a rational double point of type $E_6^0$. To characterize these derivations, we focus on the adjoint foliated structures. 
\begin{dfn}[\textup{Definition \ref{adjoint}}] 
Let $\cF$ be a foliation on $X$ and $t \in [0,1]$.
We define \textbf{the adjoint discrepancy} by
$$a(E;X,\cF,t):=ta(E;\cF)+(1-t)a(E;X).$$
The foliation $\cF$ is \textbf{$t$-log canonical} (for short, \textbf{$t$-lc}) (resp. \textbf{$t$-klt}) if $a(E;X,\cF,t) \geq -t\epE-(1-t)$ (resp. $>-t\epE-(1-t)$)
for any exceptional prime divisor $E$ over $X$. 
\end{dfn}

The condition that $\cF$ is $t$-lc provides a parameterized evaluation of the singularities of $\cF$ and $X$. In particular, the cases $t=0$ and $t=1$ correspond to $X$ being lc and $\cF$ being lc, respectively. Using this characterization and some formulas for divisors on the quotient, we prove the following theorem.
\begin{thm}[\textup{Theorem \ref{(p-1)/p}}] \label{Main Theorem1}
Let $X$ be a normal variety over $k$ and $\cF$ a 1-foliation on $X$. Then $\cF$ is $\frac{p-1}{p}$-log canonical (resp. $\frac{p-1}{p}$-klt) if and only if $X/\cF$ is log canonical (resp. klt).
\end{thm}
This theorem not only shows the relation between the singularities of foliations and those of their quotients, but also plays an important role in the classification of klt quotients of regular surfaces. We now explain how to classify the quotients.

First, we consider the blow-up of a regular variety. Then it is easy to evaluate the discrepancy $a(E;\cF)$, where $E$ is the unique exceptional divisor of the blow-up (Lemma \ref{blow-up}). In particular, for a regular surface, if the order of a generator of the foliation is $\geq 2$, then we see that it is not $\frac{2}{3}$-klt (Proposition \ref{m2}). This observation and the above theorem help us to analyze $\frac{p-1}{p}$-klt foliations.
As a result, we obtain the classification of klt quotients of regular surfaces in the cases $p=2,3$ and $5$.
\begin{thm}[\textup{Theorem \ref{p=2,3} and Corollary \ref{p=5}}] \label{Main Theorem2}
Let $S$ be a regular surface over $k$ and $\cF$ a $(p-1)/p$-klt 1-foliation on $S$.
\begin{enumerate}
    \item If $p=2$, $S/\cF$ has at worst rational double points.
    \item If $p=3$, each singular point of $S/\cF$ is either
    \begin{enumerate}
         \item a rational double point, or
         \item a toric singularity of type $\frac{1}{3}(1,1)$.
    \end{enumerate}
    \item If $p=5$, each singular point of $S/\cF$ is either 
    \begin{enumerate}
        \item a rational double point, or
        \item a toric singularity of type $\frac{1}{5}(1,1)$ or $\frac{1}{5}(1,2)$.
    \end{enumerate}
\end{enumerate}
\end{thm}
Note that a rational double point given as the quotient by a $p$-closed derivation of a regular local ring is of type $A_1,D_{2m}^0(m \geq 2),E^0_7$ and $E^0_8$ if $p=2$; $A_2,E^0_6$ and $E^0_8$ if $p=3$; and $A_4$ and $E^0_8$ if $p=5$ (for details, see \cite{Matsu1}).

This paper is organized as follows. In Section 2, we review some properties of the derivations, foliations, and quotients to prepare for later sections.
In Section 3, we introduce adjoint foliated structures and prove Theorem \ref{Main Theorem1}.
In Section 4, we specifically consider the case of regular surfaces and classify the klt quotients to obtain Theorem \ref{Main Theorem2}.

\subsection*{Acknowledgments}
I would like to thank my supervisor Takehiko Yasuda, who gave me continuous support and valuable suggestions. I am also grateful to Hiroyuki Ito, Yuya Matsumoto, and Quentin Posva for many helpful comments. 

\section{Preliminaries}
Throughout this paper, let $k$ be an algebraically closed field of characteristic $p>0$.

\subsection{Derivations}
Let $A$ be $k$-algebra. A \textbf{derivation} of $A$ over $k$ is a $k$-linear map $\delta: A \to A$ satisfying 
$$\delta(ab)=\delta(a)b+a\delta(b),\;\;\;a,b\in A.$$
We denote by $\Der{A}$ the set of all derivations of $A$ over $k$, which is an $A$-module. If $\delta,\delta' \in \Der{A}$, then we can compose $\delta$ and $\delta'$ as maps $A \to A$. Since
\begin{align*}
(\delta\delta'-\delta'\delta)(ab)&=\delta(\delta'(a)b+a\delta'(b))-\delta'(\delta(a)b+a\delta(b)) \\
&=\delta\delta'(a)b+\delta'(a)\delta(b)+\delta(a)\delta'(b)+a\delta\delta'(b)-\delta'\delta(a)b-\delta(a)\delta'(b)-\delta'(a)\delta(b)-a\delta'\delta(b) \\
&=(\delta\delta'-\delta'\delta)(a)\cdot b+a\cdot (\delta\delta'-\delta'\delta)(b),
\end{align*}
we see that the bracket $[\delta,\delta']=\delta\delta'-\delta'\delta$ is again an element of $\Der{A}$, and that $\Der{A}$ becomes a Lie algebra with this bracket. 

Also, using Leibniz formula, we obtain
$$\delta^n(ab)=\sum_{i=0}^{n} \binom{n}{i} \delta^i(a)\delta^{n-i}(b).$$
In particular, if $n=p$, then we have $\delta^p(ab)=\delta^p(a)b+a\delta^p(b)$, which implies that $\delta^p \in \Der{A}$. $\delta$ is called \textbf{$p$-closed} if there exists $\alpha \in A$ such that $\delta^p=\alpha \delta$.

\begin{dfn}
We say that $\delta$ is \textbf{multiplicative} if $\delta^p=u\delta$ for some unit $u \in A^{\times}$. We say that $\delta$ is \textbf{additive} if $\delta^p=0$.
\end{dfn}

\begin{lem} \textup{\cite[Lemma 2.3]{Matsu1}} \label{alpha}
Let $A$ be a reduced $k$-algebra and $\delta \in \Der{A}$ a $p$-closed derivation with $\delta^p=\alpha\delta$.
Then $\delta(\alpha)=0$.
\end{lem}
\begin{proof}
We have 
$$\alpha\delta^2=\delta^{p+1}=\delta \circ (\alpha\delta)=\delta(\alpha)\delta+\alpha\delta^2.$$
Hence $\delta(\alpha)\delta=0$. In particular, as $\delta(\alpha)^2=0$ holds, we obtain $\delta(\alpha)=0$. 
\end{proof}

\subsection{Foliations}
Unless stated otherwise, let $X$ be a normal variety over $k$.
We denote by $T_X$ the tangent sheaf of $X$.
\begin{dfn}
A \textbf{foliation} on $X$ is a coherent sheaf $\cF \subseteq T_X$ such that
\begin{enumerate}
\item $\cF$ is saturated in $T_X$, i.e. $T_X/\cF$ is torsion free, and
\item $\cF$ is closed under the Lie bracket.
\end{enumerate}
A foliation is called \textbf{1-foliation} if it is also closed under $p$-th powers.
\end{dfn}

Consider the exact sequence
$$0 \to \cF \to T_X \to T_X/\cF \to 0.$$
Since $T_X/\cF$ is torsion-free, $\mathrm{Ass}(T_X/\cF)$ consists only of the generic point of $X$. Also, since $T_X$ is dual of the coherent sheaf $\Omega_X$, we see that $T_X$ is reflexive. Therefore $\cF$ is also reflexive by \cite[Corollary 1.5]{Reflexive}.
The \textbf{canonical divisor} of $\cF$ is any Weil divisor $K_{\cF}$ such that $\cO_X(-K_{\cF}) \cong \text{det}(\cF)$.

\begin{dfn}
Let $\pi:Y \to X$ be a birational morphism of normal varieties and $\cF$ a 1-foliation on $X$. The stalk of $\cF$ at the generic point $\eta_X$ is a $K(X)$-subspace of $\Der{K(X)}$. Since $\pi$ induces the isomorphism $\pi^*:\Der{K(X)} \to \Der{K(Y)}$, we obtain a $K(Y)$-subspace $\mathcal{G}=\pi^*(\cF_{\eta_X})$ of $\Der{K(Y)}$. We define the \textbf{pullback foliation} $\pi^*\cF$ as the saturation of $\mathcal{G}$. 
\end{dfn}

We explain how to calculate pullback foliations along smooth blow-ups on the affine space.

\begin{ex} \label{pullback ex}
Let $A=k[[x_1,\ldots,x_n]]$ and $X=\hat{\mathbb{A}}^n_k=\mathrm{Spec}(A)$. We consider the blow-up $\pi:\tilde{X} \to X$ at the origin. Since the Rees algebra of the ideal $\fm=(x_1,\ldots,x_n)$ is given by
$$\mathcal{R}(\fm)=\bigoplus_{d \geq 0} \fm^dt^d=A[x_1t,\ldots,x_nt] \subset A[t],$$
we have $\tilde{X}=\text{Proj}(A[T_1,\ldots,T_n]/(x_iT_j-x_jT_i\mid1\leq i <j \leq n))$.
Let $\tilde{A}=\cO_{\tilde{X}}(D_+(T_1))$ be the ring of sections over $D_{+}(T_1)$. Then $\tilde{A}$ can be described as follows:
\begin{align*}
  \tilde{A}&=(A[T_1,\ldots,T_n]/(x_iT_j-x_jT_i\mid 1\leq i <j \leq n))_{(T_1)}\\
   &= A\left[\frac{T_2}{T_1},\ldots,\frac{T_n}{T_1}\right]/(x_i-\frac{T_i}{T_1}x_1 \mid 2 \leq i \leq n) \\
  &= k[[y_1,\ldots,y_n]][t_2,\ldots,t_n]/(y_i-y_1t_i\mid 2 \leq i \leq n).
\end{align*}
Here, the last equality is given by the isomorphism defined by $x_i \mapsto y_i$ and $T_i/T_1 \mapsto t_i$. Thus one chart of the blow-up is given by
$$A \to \tilde{A};\;\; (x_1, x_2,\ldots,x_n) \mapsto (y_1,y_1t_2,\ldots,y_1t_n).$$
Using the induced isomorphism
$$\pi^*:\Der{k((x,y))} \to \Der{Q(\tilde{A})};\;\; \delta \mapsto \iota \circ \delta \circ \iota^{-1},$$
where $\iota:k((x,y)) \to Q(\tilde{A})$ is the isomorphism of the quotient fields, we have
$$\pi^*\partial_{x_1}(y_1)=\iota(\partial_{x_1}(x_1))=1,\;\pi^*\partial_{x_1}(t_i)=\iota(\partial_{x_1}(\frac{x_i}{x_1}))=\iota^{-1}(-\frac{x_i}{x_1^2})=-\frac{t_i}{y_1}\;(2\leq i \leq n)$$
and for $2 \leq i \leq n$,
$$
\pi^*\partial_{x_i}(y_1)=\iota (\partial_{x_i}(x_1))=0, \;\pi^{*}\partial_{x_i}(t_i)=\iota (\partial_{x_i}(\frac{x_i}{x_1}))=
\iota(\frac{1}{x_1})=\frac{1}{y_1},\;\pi^*\partial_{x_i}(t_j)=0\;(j \neq i).
$$
After the calculations above, we find the transformation rules
\begin{equation} \label{pullback formula}
    \pi^*\partial_{x_1}=\partial_{y_1}-\sum_{i=2}^n\frac{t_i}{y_1}\partial_{t_i},\;\; \pi^*\partial_{x_i}=\frac{1}{y_1}\partial_{t_i} \;\;(2 \leq i \leq n).
\end{equation}
\end{ex}

\subsection{Quotients by 1-foliations}
We define the invariant subrings of 1-foliations and collect several lemmas on their properties.
\begin{dfn}
Let $A$ be a $k$-algebra and $\cF \subseteq \Der{A}$ a 1-foliation.
The \textbf{invariant subring} of $\cF$ is defined as follows:
$$A^{\cF} \coloneqq \{a \in A \mid \delta (a) = 0 \;\;\forall\delta \in \cF\}.$$
\end{dfn}

\begin{lem} \textup{\cite{Aramova}} \label{homeo} %1.2.1
$\mathrm{Spec}(A^{\cF})$ is canonically homeomorphic to $\mathrm{Spec}(A)$.
\end{lem}

\begin{lem} \label{quotient prop}
Let $A$ and $\cF$ be as above, and assume that $A$ is integral. Then:
\begin{enumerate}
\item $Q(A^{\cF})=Q(A)^{\cF_{(0)}}$,
\item $Q(A^\cF) \cap A=Q(A)^{\cF_{(0)}} \cap A=A^{\cF}$,
\item If $A$ is normal, so is $A^{\cF}$.
\end{enumerate}
\end{lem}
\begin{proof}
We identify $\cF_{(0)}$ with a submodule of $\Der{Q(A)}$ as follows.
For $\delta \in \Der{A}$, we can uniquely define the derivation $\bar{\delta} \in \Der{Q(A)}$ by the formula
$$\bar{\delta}(\frac{a}{b})=\frac{\delta(a)b-a\delta(b)}{b^2}.$$
Using this extension of $\delta$, we obtain the canonical map 
$$\Der{A} \to \Der{Q(A)};\;\delta \mapsto \bar{\delta}.$$
We identify $\cF_{(0)}$ with the image of $\cF$ under this map.

(1) It is easy to see that $Q(A^{\cF}) \subseteq Q(A)^{\cF_{(0)}}$.
Let $a/b \in Q(A)^{\cF_{(0)}}$. For any $\delta \in \cF$, we have
$$\bar{\delta}(\frac{a}{b})=\bar{\delta}(\frac{ab^{p-1}}{b^p})=\frac{\delta(ab^{p-1})}{b^p}=0.$$
This implies $\delta(ab^{p-1})=0$. Since $ab^{p-1},b^p \in A^{\cF}$, we obtain $$\frac{a}{b}=\frac{ab^{p-1}}{b^p} \in Q(A^{\cF}),$$
as required.

(2) Since we have $Q(A^{\cF}) \cap A=Q(A)^{\cF_{(0)}} \cap A$, it is sufficient to show that $A^{\cF} \subseteq Q(A^{\cF}) \cap A$ and $Q(A)^{\cF_{(0)}} \cap A \subseteq A^{\cF}$. The former is clear. Suppose that $a\in Q(A)^{\cF_{(0)}} \cap A$. Then it follows that  
$$\bar{\delta}(\frac{a}{1})=\frac{\delta(a)}{1}=0$$
for any $\delta \in \cF$.
Thus we have $\delta(a)=0$. Therefore it follows that $A \cap Q(A)^{\cF} \subseteq A^{\cF}$.

(3) Assume that $x \in Q(A^{\cF})$ is integral over $A^{\cF}$. Since $A^{\cF} \subseteq A$ and $A$ is normal, it follows that $x \in Q(A^{\cF}) \cap A=A^{\cF}$. Thus $A^{\cF}$ is normal. 
\end{proof}

By gluing the spectra of these rings, we can globally define quotients by 1-foliations.
\begin{dfn}
Let $\cF$ be a 1-foliation on $X$.
For an affine open cover $X=\bigcup \mathrm{Spec}(A_i)$, we define by $X/\cF=\bigcup \mathrm{Spec}(A_i^{\cF})$ the \textbf{quotient} of $X$ by $\cF$. 
\end{dfn}

\begin{rem}
A natural morphism $q:X \to X/\cF$ is induced by the ring inclusion $A^{\cF} \subseteq A$. By Lemma \ref{homeo}, this morphism gives a homeomorphism between their underlying spaces.  
\end{rem}

It is known that there is a Galois type correspondence between 1-foliations and their quotients. 

\begin{prop} \textup{\cite[Lemma 1.2.1]{MiyanishiIto}}  \label{Jacobson}
Let $P$ be a field extension of $k$. Then there exists a one-to-one correspondence between the set $\mathscr{E}$ of intermediate field extensions $\Phi$ with $P^p \subseteq \Phi \subseteq P$ and the set $\mathscr{D}$ of $p$-Lie subalgebras (i.e. closed under the Lie bracket and the $p$-th power) $\mathfrak{D}$ of $\Der{P}$.
The correspondence is given by
$$\Phi \mapsto \mathfrak{D}_{\Phi}(P)=\{\delta \in \Der{P} \mid \delta|_{\Phi}=0\}$$
and
$$\mathfrak{D} \mapsto C(\mathfrak{D})=\{\xi \in P\mid \delta(\xi)=0\; for\;all \;\delta \in \mathfrak{D}\}.$$
\end{prop}
\begin{proof}
See \cite[first paragraph on page 189]{Jacobson}.
\end{proof}

\subsection{Birational singularities of 1-foliations}
Let $\cF$ be a 1-foliation and $\Delta$ a $\QQ$-Weil divisor.
Assume that $(\cF,\Delta)$ is $\QQ$-Gorenstein (i.e. $K_{\cF}+\Delta$ is $\QQ$-Cartier). If $\pi:Y \to X$ is a birational proper $k$-morphism of normal varieties, we can write
$$K_{\pi^{*}\cF}+\pi_{*}^{-1}\Delta=\pi^{*}(K_{\cF}+\Delta)+\sum_{E} a(E;\cF,\Delta)E,$$
where $E$ runs through the $\pi$-exceptional prime divisors.

Since the discrepancy $a(E;\cF,\Delta)$ depends on whether the exceptional divisor is invariant for the pullback of $\cF$, we add a correction term $\epE$.

\begin{dfn}
Let $\cF$ be a 1-foliation on $X$.
\begin{enumerate}
\item  A prime divisor $E \subset X$ is called \textbf{$\cF$-invariant} if $\cF(I_E) \subseteq I_E$ at the generic point of $E$.
\item For any birational morphism $\pi:Y \to X$ with $Y$ normal and any prime divisor $E \subset Y$, we define 
$$\epE=
\begin{cases*}
0  & if $E$ is $\pi^*\cF$-invariant, \\
1  & otherwise.
\end{cases*}
$$
This does not depend on the choice of $\pi$.
\end{enumerate}
\end{dfn}

\begin{dfn}
Suppose $(\cF,\Delta)$ is $\QQ$-Gorenstein.
\begin{enumerate}
 \item $(\cF,\Delta)$ is \textbf{terminal} (resp. \textbf{canonical}) if $a(E;\cF,\Delta) >0$ (resp. $a(E;\cF,\Delta) \geq 0$) for all exceptional prime divisors $E$ over $X$;
\item $(\cF,\Delta)$ is \textbf{klt}  if $\lfloor \Delta \rfloor =0$ and $a(E;\cF,\Delta) > -\epE$ for all exceptional $E$ over $X$;
\item $(\cF,\Delta)$ is \textbf{log canonical} (\textbf{lc}) if $a(E;\cF,\Delta) \geq - \epE$ for all exceptional $E$ over $X$.
\end{enumerate}
We say that $\cF$ is terminal (resp. canonical, klt, lc) if $(\cF,0)$ is terminal (resp. canonical, klt, lc).
\end{dfn}

For a 1-foliation $\cF$ on a regular variety, we can easily estimate the discrepancy $a(E;\cF)$ along the exceptional divisor $E$ of its blow-up.

\begin{lem}
\label{blow-up}
Let $X$ be a regular variety and $\cF$ an invertible 1-foliation on $X$ generated by $\delta$ at $q \in X$. Let $\pi:\tilde{X} \to X$ be the blow-up at $q$ and $E$ the unique exceptional divisor. 
If $\delta \in \fm^d\Der{\cO_{X,q}}$, then $a(E;\cF) \leq -d+1-\epE$.
\end{lem}
\begin{proof}
We may assume that $X=\hat{\mathbb{A}}^n_k$. In what follows, we use the notation of Example \ref{pullback ex}. Suppose that $\delta$ is written as
$$\delta=\sum_{i=1}^n f_i(\mathbf{x})\partial_{x_i},$$
with $f_i(\mathbf{x})=f_i(x_1,\ldots,x_n) \in \fm^d$.
Using the transformation rules \eqref{pullback formula}, we have
\begin{equation} \label{pi*}
\begin{split}
\pi^*\delta&=\tilde{f_1}(\mathbf{y})(\partial_{y_1}-\sum_{i=2}^n\frac{t_i}{y_1}\partial_{t_i})+\sum_{i=2}^n\tilde{f_i}(\mathbf{y}) \cdot\frac{1}{y_1}\partial_{t_i}\\
&=\tilde{f_1}(\mathbf{y})\partial_{y_1}+\sum_{i=2}^n\frac{1}{y_1}(\tilde{f_i}(\mathbf{y})-t_i\tilde{f_1}(\mathbf{y}))\partial_{t_i} \\
&=y_1^{d-1}\left[y_1 \cdot \frac{\tilde{f_1}(\mathbf{y})}{y_1^d}\partial_{y_1}+\sum_{i=2}^{n}\frac{\tilde{f_i}(\mathbf{y})-t_i\tilde{f_1}(\mathbf{y})}{y_1^d}\partial_{t_i}\right],
\end{split}
\end{equation}
where $\tilde{f_i}(\mathbf{y})=f_i(y_1,y_1t_2,\ldots,y_1t_n)$. Since $f_i(\mathbf{x}) \in \fm^d$, we have $y_1^{-d}\tilde{f_i}(\mathbf{y}) \in \tilde{A}$.

Let $I_E=(y_1)$ be the defining ideal of $E$ and write $\pi^*\delta=y_1^e\tilde{\delta}$, where $\tilde{\delta} \in \Der{\tilde{A}} \setminus I_E\Der{\tilde{A}}$. Then we can calculate 
$$\pi^*K_\cF=\text{div}(\pi^*\delta)=\text{div}(y_1^e\tilde{\delta})=K_{\pi^*\cF}+eE.$$
From the above calculation, we see that $a(E;\cF)=-e$.

If $E$ is $\pi^*\cF$-invariant, then $e \geq d-1=d-1+\epE$. If $E$ is not $\pi^*\cF$-invariant, then it follows that $\tilde{\delta}(y_1) \notin I_E$. This implies that the derivation in brackets must be divisible by $y_1$, since the coefficient of $\partial_{y_1}$ belongs to $I_E$. Hence we have $e \geq d=d-1+\epE$.
In both cases, we obtain
\begin{equation*}
    a(E;\cF)=-e \leq -d+1-\epE. \tag*{\qedhere}
\end{equation*}
\end{proof}

\section{Adjoint foliated structures and their quotients}
In this section, we introduce the adjoint foliated structures to characterize the 1-foliations whose quotients are klt or log canonical. From now on, we assume that $X$ is $\QQ$-Gorenstein.
\subsection{Adjoint foliated structures}
\begin{dfn} \textup{\cite[Definition 2.6,2.7]{Adjoint}} \label{adjoint}
Let $\cF$ be a foliation on $X$ and let $t \in [0,1]$.
We define the adjoint discrepancy by
$$a(E;X,\cF,t):=ta(E;\cF)+(1-t)a(E;X).$$
The foliation $\cF$ is \textbf{$t$-log canonical} (for short, \textbf{$t$-lc}) (resp. \textbf{$t$-klt}) if $a(E;X,\cF,t) \geq -t\epE-(1-t)$ (resp. $>-t\epE-(1-t)$) for any exceptional prime divisor $E$ over $X$.
\end{dfn}

\begin{lem} \label{t'-klt}
If $X$ is log canonical (resp. klt) and $\cF$ is $t_0$-lc (resp. $t_0$-klt), then  $\cF$ is $t$-lc (resp. $t$-klt) for any $t \leq t_0$.
\end{lem}
\begin{proof}
Fix an arbitrary exceptional divisor $E$ over $X$. 
We define
$$F(t):=(a(E;\cF)-a(E;X)+\epE-1)t+a(E;X)+1.$$
Then the assumption implies $F(0) \geq 0$ (resp. $F(0) > 0$) and $F(t_0) \geq 0$ (resp. $F(t_0) > 0$). By the property of linear functions, we have $F(t)\geq0$ (resp. $F(t) >0$) for any $t \in [0,t_0]$. Thus $\cF$ is $t$-lc (resp. $t$-klt) for any $t \in [0,t_0]$.
\end{proof}

\subsection{Birational singularities of quotients}
We recall two well-known formulas for quotients.
\begin{lem} \textup{\cite[Proposition 1]{Insep}} \label{pullback E'}
Let $\cF$ be a 1-foliation on $X$ and $q:X \to Y$ its quotient. For a prime divisor $E \subset X$ with image $q(E)=E' \subset Y$:
\begin{enumerate}
\item if $E$ is $\cF$-invariant, then $q^*E'=E$;
\item if $E$ is not $\cF$-invariant, then $q^*E'=pE$.
\end{enumerate}
\end{lem}

\begin{lem} \textup{\cite[Propsition 1.1]{Ekedahl}} \label{adjunction}
Let $q:X \to X/\cF$ be the quotient by a 1-foliation $\cF$. Then we have the following equality of Weil divisors:
$$q^*K_{X/\cF}=K_X+(p-1)K_\cF.$$
\end{lem}

\begin{rem}
If $X$ and $\cF$ are $\QQ$-Gorenstein, so is $X/\cF$ by \cite[Lemma 4.9]{PosvaA}.
\end{rem}

We discuss the correspondence between birational transformations of foliations and those of their quotients, and the relationship between their discrepancies.
\begin{prop}\textup{\cite[Theorem 2.10]{PosvaB}} \label{rel}
Let $\cF$ be a $\QQ$-Gorenstein 1-foliation on $X$ and $q:X \to Y$ its quotient.
\begin{enumerate}
\item For any proper birational morphism $\nu:X' \to X$ with $X'$ normal, there exists a proper birational morphism $\mu:Y' \to Y$ such that $Y'=X'/\nu^*\cF$ and diagram \eqref{diagram} below commutes. 
\begin{equation} \label{diagram}
\begin{tikzpicture}[auto,x=12mm,y=12mm,->,baseline=(current bounding box.center)]
\node (a) at (0,1) {$X'$}; \node (x) at (1,1) {$X$};
\node (b) at (0,0) {$Y'$}; \node (y) at (1,0) {$Y$};
\draw (a)-- node
{$\scriptstyle \nu$} (x);
\draw (x)-- node
{$\scriptstyle q$} (y);
\draw (a)-- node[swap] {$\scriptstyle q'$} (b);
\draw (b)-- node[swap] {$\scriptstyle \mu$} (y);
\end{tikzpicture}
\end{equation}
\item Conversely, for any proper birational morphism $\mu:Y' \to Y$ with $Y'$ normal, there exists a proper birational morphism $\nu:X' \to X$ such that $X'$ is normal and diagram \eqref{diagram} above commutes.
Moreover, in this setting, $q'$ is the quotient by $\nu^*\cF$.
\item If $E \subset X'$ is a $\nu$-exceptional divisor with image the $\mu$-exceptional divisor $E'=q'(E) \subset Y'$, then
\begin{equation} \label{discrepancy relation}
a(E';X/\cF)\delta(E)=(p-1)a(E;\cF)+a(E;X),
\end{equation}
where
$$\delta(E)=
\begin{cases*}
1  & if $E$ is $\nu^*\cF$-invariant,\\
p  & otherwise.
\end{cases*}
$$
\end{enumerate}
\end{prop}
\begin{proof}
(1) Let $U=\text{Spec}\;A \subseteq X$ and $V=\text{Spec}\;B \subseteq X'$ be affine open subsets such that $V \subseteq\nu^{-1}(U)$. We denote by $f:A \to B$ the homomorphism induced by $\nu$. By the definition of pullback foliation, the stalk of $\nu^*\cF$ at $\eta_Y$ is given by
$$(\nu^*\cF)_{(0)}=\{\bar{f}\delta\bar{f}^{-1} \in \Der{Q(B)} \mid \delta \in \cF_{(0)}\},$$
where $\bar{f}:Q(A) \to Q(B)$ is the isomorphism induced by $f$.
For any $x \in A^{\cF}$ and $\delta \in \cF_{(0)}$, we have
$$(\bar{f}\delta\bar{f}^{-1})(f(x))=\bar{f}\delta(x)=0.$$
Thus $f(x) \in \Frac(B)^{\nu^*\cF} \cap B=B^{\nu^*\cF}$ by (2) of Lemma \ref{quotient prop}. Thus we define the homomorphism
$$g:A^{\cF} \to B^{\nu^*\cF};\;\; x \mapsto f(x).$$
Gluing the morphisms induced by these homomorphisms, we obtain the desired morphism $\mu:X'/\nu^*\cF \to X/\cF$.

(2) Let $\mu:Y' \to Y$ be a proper birational morphism with $Y'$ normal. We consider the normalization of the fiber product $X \times_Y Y'$. Let $p_1:X \times_Y Y' \to X$ and $p_2:X \times_Y Y' \to Y'$ be projection morphisms and $\nu_0:X' \to X \times_Y Y'$ the normalization. 

Since $\mu$ is proper, so is $p_1$. 
Since $q$ is a universal homeomorphism, $p_2$ is a homeomorphism.
In addition, the normalization $\nu_0$ is finite and a homeomorphism.
Let $\nu \coloneq p_2 \circ \nu_0$ and $q' \coloneq p_1 \circ \nu_0$. Then $\nu$ is a proper birational morphism and $q'$ is a purely inseparable morphism. 

We will show that the morphism $q':X' \to Y'$ is the quotient of $\nu^*\cF$. By Proposition \ref{Jacobson}, it is sufficient to show that 
\begin{equation} \label{KY}
    K(Y')=K(X')^{(\nu^*\cF)_{\eta_{X'}}}.
\end{equation}
By definition, the pullback of $\cF$ by $\nu$ is given by  
$$(\nu^*\cF)_{\eta_{X'}} = \{\ \nu_{\eta}\circ\delta\circ\nu_{\eta}^{-1} \mid \delta \in \cF_{\eta_X} \} \subseteq \Der{K(X')},$$
where $\nu_{\eta}:K(X) \to K(X')$ is the isomorphism between the function fields induced by $\nu$.
Then, for $f \in K(X')$, we have
\begin{align*}
f \in K(Y') &\iff \nu_{\eta}^{-1}(f) \in K(Y)=K(X')^{\cF_{\eta_X}} \\
&\iff  \delta(\nu_{\eta}^{-1}(f))=0 \;\; \forall\delta \in \cF_{\eta_X}\\
&\iff  (\nu_{\eta}\circ\delta\circ\nu_{\eta}^{-1})(f)=0 \;\; \forall\delta \in \cF_{\eta_X} \\
&\iff f \in K(X')^{(\nu^*\cF)_{\eta_{X'}}}.
\end{align*}
Thus we obtain equation \eqref{KY}.

(3) Using the adjunction formula (Lemma \ref{adjunction}), we have
\begin{align}
K_X&=q^*K_Y-(p-1)K_{\cF}, \label{ad1}\\
K_{X'}&=q'^*K_{Y'}-(p-1)K_{\nu^*\cF}. \label{ad2}
\end{align}
In addition, by the definition of discrepancies, the following equalities hold:
\begin{align}
K_{Y'}&=\mu^*K_Y+\sum a(E';Y)E', \label{disc1}\\
K_{X'}&=\nu^*K_X+\sum a(E;X)E, \label{disc2}\\
\nu^*K_{\cF}&=K_{\nu^*\cF}-\sum a(E;\cF)E. \label{folidisc}
\end{align}
Substituting \eqref{disc2} into equation \eqref{ad2}, we obtain
\begin{equation} \label{KX'a}
\begin{split}
K_{X'}&=q'^*(\mu^*K_Y+\sum a(E';Y)E')-(p-1)K_{\nu^*\cF} \\
&=q'^*\mu^*K_Y-(p-1)K_{\nu^*\cF}+\sum a(E';Y)\delta(E)E.
\end{split}
\end{equation}
Here, we used $q'^*E'=\delta(E)E$, which follows from Lemma \ref{pullback E'}.
Moreover, substituting \eqref{ad1} and \eqref{folidisc} successively into equation \eqref{disc1}, we have
\begin{equation} \label{KX'b}
\begin{split}
K_{X'}&=\nu^*(q^*K_Y-(p-1)K_\cF)+\sum a(E;X)E \\
&=\nu^*q^*K_Y-(p-1)\nu^*K_\cF+\sum a(E;X)E \\
&=\nu^*q^*K_Y-(p-1)(K_{\nu^*\cF}-\sum a(E;\cF)E)+\sum a(E;X)E \\
&=\nu^*q^*K_Y-(p-1)K_{\nu^*\cF}+\sum ((p-1)a(E;\cF)+a(E;X))E.
\end{split}
\end{equation}
Since $q'\mu=\nu q$, we can compare the coefficient of $E$ in \eqref{KX'a} and \eqref{KX'b} to obtain
$$a(E';Y)\delta(E)=(p-1)a(E;\cF)+a(E;X).$$
\end{proof}

Using the commutative diagram and the formula in Proposition \ref{rel}, we obtain the following main theorem.

\begin{thm} \label{(p-1)/p}
Let $\cF$ be a $\QQ$-Gorenstein 1-foliation on $X$. Then $\cF$ is $\frac{p-1}{p}$-log canonical (resp. $\frac{p-1}{p}$-klt) \\ if and only if $X/\cF$ is log canonical (resp. klt).
\end{thm}
\begin{proof}
By (1) and (2) of Proposition \ref{rel}, any proper birational morphism $\nu:X' \to X$ induces a proper birational morphism $\mu:Y' \to Y$ making diagram \eqref{diagram} commute, and conversely, any morphism $\mu$ induces a morphism $\nu$.
Let $E \subset X'$ be a $\nu$-exceptional divisor and let $E'=q'(E) \subset Y'$ be the $\mu$-exceptional divisor corresponding to $E$. 
Note that
\begin{equation} \label{delta}
   \delta(E)=(p-1)\epE+1.
\end{equation}
Using \eqref{discrepancy relation} and \eqref{delta}, we have
\begin{align*}
    &\frac{p-1}{p}a(E;\cF)+\frac{1}{p}a(E;X) \geq -\frac{p-1}{p}\epE-\frac{1}{p} \\
    &\iff \frac{1}{p}a(E';X/\cF)\delta(E) \geq -\frac{1}{p}\delta(E) \\
    &\iff a(E';X/\cF) \geq -1.
\end{align*}
Thus $\cF$ is $\frac{p-1}{p}$-log canonical if and only if $X/\cF$ is log canonical.
Similarly for the equivalence between klt properties.
\end{proof}

\section{Classification of klt surface quotients}

In this section, we classify the klt quotients by 1-foliations on regular surfaces. 

\subsection{Some auxiliary results}
Before the classification, we introduce some useful propositions.

\begin{prop} \textup{\cite[Proposition 2.28]{Kollar}} \label{rational} Let $X$ be a normal surface and let $\Delta$ be an effective $\mathbb{R}$-divisor. If $(X,\Delta)$ is klt, then $X$ has at worst rational singularities.
\end{prop}

\begin{prop} \textup{\cite[Proposition 3.6]{PosvaA}} \label{nilpotence}
Let $x \in X$ be a closed point of a regular variety and $\cF$ an invertible 1-foliation on $X$ generated by $\delta$ at $x$. Then:
\begin{enumerate}
\item if $\delta \notin \fm \Der{\cO_{X,x}}$, then $\cF$ is canonical at $x \in X$;
\item if $\delta \in \fm \Der{\cO_{X,x}}$, then $\cF$ is log canonical at $x\in X$ if and only if the induced endomorphism $\delta|_{\fm/\fm^2}$ is not nilpotent.
\end{enumerate}
\end{prop}

In what follows, unless otherwise stated, let $S$ be a regular surface over $k$.

\begin{thm} \textup{\cite[Theorem 4.19]{PosvaA}} \label{Posva}
Let $\cF$ be a 1-foliation of rank 1 on $S$. Then $S/\cF$ is F-regular if and only if $\cF$ is log canonical.
\end{thm}

\begin{thm} \textup{\cite[Theorem 2.3]{Hirokado}} \label{multi}
Let $\cF$ be a 1-foliation on $S$. If $\cF$ is generated by a multiplicative derivation at $q \in S$, then there exists a system of parameters $x,y \in \widehat{\cO}_{S,q}$ such that the singularity on the quotient surface $V=S/\cF$ at $\tilde{q}$ is expressed as
$$\widehat{\cO}_{V,\tilde{q}} \cong k[[x,y]]^\delta,\;\;\delta=x\dd{x}+\lambda y\dd{y},\;(\lambda=1,2,\cdots,p-1).$$
Moreover, for $x,y$ as above, the singularity of $V$ at $\tilde{q}$ is a toric singularity of type $\frac{1}{p}(1,\lambda)$:
$$k[[x,y]]^{\delta} = k[[x^iy^j \mid i,j \geq 0, \;\;i+\lambda j \equiv 0 \;\;\mathrm{mod}\;p]].$$
\end{thm}

By \cite{LRQ}, it is known that F-regularity is equivalent to being klt if $p \geq 7$ in dimension $2$. Thus, using Theorem \ref{Posva} and Theorem \ref{multi}, we complete the classification of klt quotients of regular surfaces in the case $p \geq 7$. So we consider the cases $p=2,3$ and $5$. 

The following proposition plays an important role in the classification.

\begin{prop} \label{m2}
Let $\cF$ be a 1-foliation on $S$ generated by $\delta \in \Der{\cO_{S,q}}$ at $q \in S$. 
If $\delta \in \fm^2\Der{\cO_{S,q}}$, then $\cF$ is not $\frac{2}{3}$-klt.
\end{prop}
\begin{proof}
Let $\pi : \tilde{S} \to S$ be the blow-up at $q$ and $E$ the unique exceptional divisor. Then we have $a(E;S)=1$.
By Lemma \ref{blow-up}, the assumption of $\delta$ implies $a(E;\cF) \leq -1-\epE$. Hence we obtain
$$\frac{2}{3}a(E;\cF)+\frac{1}{3}a(E;S) \leq \frac{2}{3}(-1-\epE)+\frac{1}{3}=-\frac{2}{3}\epE-\frac{1}{3}.$$
Thus $\cF$ is not $\frac{2}{3}$-klt.
\end{proof}

\subsection{Case $p=2,3$}
Let $\delta$ be a $p$-closed derivation of $k[[x,y]]$. Then $\delta$ induces the invariant ring
$$ A: = k[[x,y]]^{\delta}=\{a \in k [[x,y]] \mid \delta (a) = 0\}.$$
This ring is a complete normal $k$-subalgebra of $k[[x,y]]$ containing $k[[x^p,y^p]]$ and the quotient field $K=Q(A)$ is a purely inseparable extension of $k((x^p,y^p))$ of degree $p$.

\begin{prop} \textup{\cite{Miyanishi}} \label{MR}
In the above setting, let $\fm$ be the maximal ideal of $k[[x,y]]$. 
\begin{enumerate}
\item If $p=2$, there exists an element $\psi\in \fm$ such that $A=k[[x^2,y^2,\psi]]$.
\item If $p=3$, there exist elements $\phi,\psi$ of $k[[x,y]]$ such that $A=k[[x^3,y^3,\phi^2\psi,\phi\psi^2]]$. Furthermore, in this situation, the following hold:
\begin{enumerate}
    \item at least one of $\phi$ and $\psi$ belongs to $\fm$; and
    \item $1,\phi^2\psi,\phi\psi^2$ form a free $k[[x^3,y^3]]$-basis of $A$.
\end{enumerate} 
\end{enumerate}
\end{prop}
\begin{proof}
Let $B=k[[x^p,y^p]]$. By \cite[Proposition 2.2.11]{CM}, we see that $A$ is a free $B$-module of rank $p$. We can choose a $B$-basis of $A$ containing $1$ as a member.
If $p=2$, then there exists $\psi \in \fm$ such that $1,\psi$ form a free $B$-basis of $A$. Since $B=A^2$, we have $\psi^2 \in B$, showing that $k[[x^2,y^2]][\psi] \subseteq A$. Since $A$ is complete, we obtain $k[[x^2,y^2,\psi]] \subseteq A \subseteq k[[x^2,y^2,\psi]]$, as required.

Suppose $p=3$. We write $f=\sum_{\nu=0}^{\infty} f_{\nu}$, where $f_{\nu}$ denotes the degree $\nu$ part of $f$. For convenience, we set $f_{\nu}=0$ for $\nu < 0$. Moreover, note that $f_0 \in k$.
We can choose $\tau_1,\tau_2\in A$ satisfying the following conditions, where we set $\nu_1 \coloneq \mathrm{ord}(\tau_1)$ and $\nu_2 \coloneq \mathrm{ord}(\tau_2)$:
\begin{enumerate}[label=(\roman*)]
    \item $0<\nu_1 \leq \nu_2$;
    \item $1,\tau_1,\tau_2$ form a $B$-basis of $A$;
    \item $\tau_{1,\nu_1} \notin B$;
    \item $\tau_{2,\nu_2} \notin B+B\tau_{1,\nu_1}$.
\end{enumerate}
Write
\begin{align}
    \tau_1^2&=a_0+a_1\tau_1+a_2\tau_2, \label{tau_1^2}\\
    \tau_2^2&=b_0+b_1\tau_1+b_2\tau_2, \label{tau_2^2}\\
    \tau_1\tau_2&=c_0+c_1\tau_1+c_2\tau_2, \label{tau_1tau_2}
\end{align}
with $a_i,b_i,c_i \in B$. 

Now we will show that $a_1,b_2 \in \fm$.
First, comparing the terms of order $\nu_1$ on both sides of \eqref{tau_1^2}, we have
\begin{equation} \label{a}
    a_{0,\nu_1}+a_{1,0}\tau_{1,\nu_1}+a_{2,\nu_1-\nu_2}\tau_{2,\nu_2}=0.
\end{equation}
If $a_{2,0} \neq 0$ in the case $\nu_1=\nu_2$, then we have $\tau_{2,\nu_2}=-a_{2,0}^{-1}(a_{0,\nu_1}+a_{1,0}\tau_{1,\nu_1})$, which contradicts condition $(\mathrm{iv})$. Thus it follows that $a_{2,0}=0$. Hence equation \eqref{a} becomes
$$a_{0,\nu_1}+a_{1,0}\tau_{1,\nu_1}=0.$$
If $a_{1,0} \neq 0$, then we have $\tau_{1,\nu_1}=-a_{1,0}^{-1}a_{0,\nu_1}$, which contradicts condition $(\mathrm{iii})$. Thus it follows that $a_{1,0}=0$. Therefore we obtain $a_1 \in \fm$.

Next, we proceed by induction to show that $b_{1,\nu}=0$ for $0 \leq \nu < \nu_2-\nu_1$.
Assume that the assertion holds for $0 \leq l < \nu$. Comparing the terms of order $\nu_1 + \nu$ on both sides of \eqref{tau_2^2}, we have
\begin{equation} \label{b}
    b_{0,\nu_1+\nu}=-b_{1,\nu}\tau_{1,\nu_1}.
\end{equation}
Suppose that $b_{1,\nu} \neq 0$. We compare the exponent of each prime factor on both sides of \eqref{b}. Since $b_{0,\nu_1+\nu},b_{1,\nu} \in B$, the exponent of each prime factor of $\tau_{1,\nu_1}$ must be a multiple of $3$. Hence we have $\tau_{1,\nu_1} \in B$, which contradicts condition $(\mathrm{iii})$. Thus it follows that $b_{1,\nu}=0$. Finally, comparing the terms of order $\nu_2$ on both sides of \eqref{tau_2^2}, we have
$$b_{0,\nu_2}+b_{1,\nu_2-\nu_1}\tau_{1,\nu_1}+b_{2,0}\tau_{2,\nu_2}=0.$$
If $b_{2,0} \neq 0$, then we have $\tau_{2,\nu_2}=-b_{2,0}^{-1}(b_{0,\nu_2}+b_{1,\nu_2-\nu_1}\tau_{1,\nu_1})$, which contradicts condition $(\mathrm{iv})$. Therefore we obtain $b_2 \in \fm$. By replacing $\tau_1$ with $\tau_1+a_1$ and $\tau_2$ with $\tau_2+b_2$, we may assume that $a_1=b_2=0$ while maintaining the condition that $\tau_1,\tau_2 \in \fm$.

By \eqref{tau_1^2}-\eqref{tau_1tau_2}, $\tau_1^3$ and $\tau_2^3$ are written as
\begin{align*}
\tau_1^3=\tau_1(a_0+a_2\tau_2)&=a_2c_0+(a_0+a_2c_1)\tau_1+a_2c_2\tau_2,\\
\tau_2^3=\tau_2(b_0+b_1\tau_1)&=b_1c_0+b_1c_1\tau_1+(b_0+b_1c_2)\tau_2.
\end{align*}
Since $B=A^3$, it follows that $\tau_1^3,\tau_2^3\in B$. Thus we have $a_2c_2=0$ and $b_1c_1=0$. Since $\tau_1^2,\tau_2^2 \notin B$, we see that $a_2 \neq 0$ and $b_1 \neq 0$. Hence we have $c_1=c_2=0$. 
Again, since $\tau_1^3 \in B$, the coefficient $a_0+a_2c_1=a_0$ of $\tau_1$ in $\tau_1^3$ must be $0$. Similarly, $b_0=0$. 
Therefore we obtain
\begin{equation} \label{tau}
    \tau_1^2=a_2\tau_2, \;\;\tau_2^2=b_1\tau_1,\;\;\text{and}\;\;\tau_1\tau_2=c_0, 
\end{equation}
where $c_0=a_2b_1$. Write $a_2=\phi^3$ and $b_1=\psi^3$, where $\phi,\psi \in k[[x,y]]$. By the expressions \eqref{tau}, we have $$\left(\frac{\tau_1}{\tau_2}\right)^3=\frac{a_2\tau_2\cdot\tau_1}{b_1\tau_1\cdot\tau_2}=\left(\frac{\phi}{\psi}\right)^3,$$ and hence $\tau_1/\tau_2=\phi/\psi$. Using \eqref{tau} again, we obtain
$$\tau_1=a_2\cdot\frac{\tau_2}{\tau_1}=\phi^3\cdot\frac{\psi}{\phi}=\phi^2\psi\;\;\text{and}\;\;\tau_2=b_1\cdot\frac{\tau_1}{\tau_2}=\psi^3\cdot\frac{\phi}{\psi}=\phi\psi^2.$$
Moreover, since $\tau_1,\tau_2 \in \fm$, at least one of $\phi$ and $\psi$ belongs to $\fm$. Therefore we get the desired description of $A$.
\end{proof}

To prove the following lemma, we identify the ring $k[[u,v]]$ with the base change of $k[[x,y]]$ by the Frobenius map of $k$ via the isomorphism
$$k[[x,y]] \otimes_k k \to k[[u,v]];\;\;  x \otimes 1 \mapsto u,\;y \otimes 1 \mapsto v$$
and define the Frobenius twist
$F:k[[x,y]] \to k[[u,v]]$ as the composition of the canonical homomorphism
$$k[[x,y]] \to k[[x,y]] \otimes_{k} k;\;\; f(x,y) \mapsto f(x,y) \otimes 1$$
and the above identification. From now on, we denote $F(f(x,y))$ by $f^{(p)}(u,v)$. In concrete terms, if $f(x,y)$ is written as
$$f(x,y)=\sum_{i,j}a_{ij}x^iy^j,$$
then
$$f^{(p)}(u,v)=\sum_{i,j}a_{ij}^pu^iv^j.$$
In addition, substituting $x^p$ and $y^p$ for $u$ and $v$, respectively, we have
\begin{equation} \label{(p)}
    f^{(p)}(x^p,y^p)=\sum_{i,j} a_{ij}^px^{pi}y^{pj}=f(x,y)^p.
\end{equation}
\begin{lem} \label{coprime}
Suppose $p=3$. For $\phi$ and $\psi$ as in (2) of Proposition \ref{MR}, we have 
$$\mathrm{gcd}(\phi,\psi)=1 \;\; \text{and} \;\; \mathrm{gcd}(\phi_x\psi-\phi\psi_x,\phi_y\psi-\phi\psi_y)=1.$$
\end{lem}
\begin{proof}
Let $\tau_1=\phi^2\psi$ and $\tau_2=\phi\psi^2$. We consider the surjective homomorphism
$$\theta:k[[u,v,X,Y]] \to A; \;\;(u,v,X,Y) \mapsto (x^3,y^3,\tau_1,\tau_2).$$
Using equation \eqref{(p)}, we see that
\begin{align*}
    \theta(X^2-\Fh Y)&=(\phi^2\psi)^2-\phi^3\cdot\phi\psi^2=0, \\
    \theta(Y^2-\Fs X)&=(\phi\psi^2)^2-\psi^3\cdot\phi^2\psi=0, \\
    \theta(XY-\Fh\Fs)&=\phi^2\psi\cdot\phi\psi^2-\phi^3\psi^3=0.
\end{align*}
Thus $\theta$ induces the homomorphism 
$$\bar{\theta}:k[[u,v,X,Y]]/I \to A,\;\; \text{where}\; I=(X^2-\Fh Y,Y^2-\Fs X, XY-\Fh\Fs).$$ Let $\bar{f} \in \text{Ker}\;\bar{\theta}$. Using the congruences $X^2 \equiv \Fh Y, Y^2 \equiv \Fs X$, and $XY \equiv \Fh\Fs \pmod{I}$, we may assume that $f=f_0(u,v)+f_1(u,v)X+f_2(u,v)Y$ with $f_0(u,v),f_1(u,v),f_2(u,v) \in k[[u,v]]$. Then the assumption implies $$\theta(f)=f_0(x^3,y^3)+f_1(x^3,y^3)\tau_1+f_2(x^3,y^3)\tau_2=0.$$
Since $1,\tau_1,\tau_2$ form a free $k[[x^3,y^3]]$-basis of $A$ by Proposition \ref{MR}, we have $f_0=f_1=f_2=0$. Thus $\bar{\theta}$ is an isomorphism.

We prove by contradiction that $\mathrm{gcd}(\phi,\psi)=1$ and $\mathrm{gcd}(\phi_x\psi-\phi\psi_x,\phi_y\psi-\phi\psi_y)=1$. Let $\fp=(\tau)$ be a prime ideal of height $1$ of $k[[x,y]]$, and let $P$ be the prime ideal of $k[[u,v,X,Y]]/I$ induced by $\fp$. By Lemma \ref{homeo}, we see that $\text{ht}(P)=1$.

($\mathrm{i}$) Suppose that $\tau$ is an irreducible factor of both $\phi$ and $\psi$. Then it follows that $\Fh,\Fs \in P$. Since $X^2-\Fh Y,Y^2-\Fs X \in P$, we have $X^2,Y^2 \in P$. As $P$ is a prime ideal of $k[[u,v,X,Y]]/I$, we have $X,Y \in P$. 
Thus we obtain $x,y \in \fp$, which contradicts the assumption that $\text{ht}(\fp)=1$.

($\mathrm{ii}$) Suppose that $\tau$ is an irreducible factor of both $\phi_x\psi-\phi\psi_x$ and $\phi_y\psi-\phi\psi_y$. Since the Frobenius twist and the derivations commute, we have
$$(\phi_x\psi-\phi\psi_x)^{(3)}=\Fh_u\Fs-\Fh\Fs_u\;\; \text{and} \;\; (\phi_y\psi-\phi\psi_y)^{(3)}=\Fh_v\Fs-\Fh\Fs_v.$$
Thus the assumption implies $\Fh_u\Fs-\Fh\Fs_u,\Fh_v\Fs-\Fh\Fs_v \in P$. From ($\mathrm{i}$), we may assume that $\Fh \notin P$. Then the ideal $P$ induces a prime ideal $Q$ of $A' \coloneq (k[[u,v,X,Y]]/I)_{\Fh}$.

On the other hand, we have $A' \cong k[[u,v,Z]]_{\Fh}/(Z^3-\Fs/\Fh)$.
Also, the Jacobian matrix $J$ of $A'$ is
$$J=\begin{pmatrix}
    \dfrac{\Fh_u\Fs-\Fh\Fs_u}{(\Fh)^2} & \dfrac{\Fh_v\Fs-\Fh\Fs_v}{(\Fh)^2} & 0
\end{pmatrix}.$$
Since $\Fh_u\Fs-\Fh\Fs_u,\Fh_v\Fs-\Fh\Fs_v \in Q$, it follows that
$$\text{rank}\;J \;(\text{mod} \;Q)  = 0.$$
By Jacobian criterion for regularity, $A'_Q=(k[[u,v,X,Y]]/I)_P$ is not regular. Since $P$ is a prime ideal of height $1$, $k[[u,v,X,Y]]/I$ does not satisfy Serre's condition $(R_1)$. This contradicts the normality of $A$.
\end{proof}

Using the expressions of $k[[x,y]]^{\delta}$ given in Proposition \ref{MR}, we classify the klt quotients of regular surfaces in the case of $p=2,3$.

\begin{thm} \label{p=2,3}
 Suppose that $\cF$ is a $\frac{p-1}{p}$-klt 1-foliation on $S$. 
\begin{enumerate}
\item If $p=2$, then $S/\cF$ has at worst rational double points.
\item If $p=3$, then each singular point of $S/\cF$ is either
\begin{enumerate}
    \item a rational double point, or
    \item a toric singularity of type $\frac{1}{3}(1,1)$.
\end{enumerate}
\end{enumerate}
\end{thm}
\begin{proof}
We may assume that $S=\text{Spec}(\cO)$ is the spectrum of a regular local 2-dimensional ring. Let $A=\cO^{\cF}$ be the invariant subring of $\cF$. Then we have $\widehat{A}\cong k[[x,y]]^{\delta}$, where $\delta \in \Der{k[[x,y]]}$ is a generator of $\cF$. Suppose $p=2$. By Proposition \ref{MR}, we have the expression $$\widehat{A}=k[[x^2,y^2,\psi]].$$ 
Since $\psi^2 \in k[[x^2,y^2]]$, $\widehat{A}$ is a hypersurface singularity. In particular, $\widehat{A}$ is Gorenstein. By Theorem \ref{(p-1)/p} and Proposition \ref{rational}, $\widehat{A}$ has at worst rational singularities. Therefore $\widehat{A}$ is a rational double point.

Suppose $p=3$. Let $\widehat{A}=k[[x^3,y^3,\phi^2\psi,\phi\psi^2]]$ be the expression of $\widehat{A}$ given in (2) of Proposition \ref{MR}. We define 
$$\delta'=(\phi_y\psi-\phi\psi_y)\partial_x-(\phi_x\psi-\phi\psi_x)\partial_y.$$ Then we see that
$$\delta'(\frac{\phi}{\psi})=(\phi_y\psi-\phi\psi_y)\cdot\frac{\phi_x\psi-\phi\psi_x}{\psi^2}-(\phi_x\psi-\phi\psi_x)\cdot\frac{\phi_y\psi-\phi\psi_y}{\psi^2}=0,$$
which implies $\phi/\psi \in k((x,y))^{\delta'}$. On the other hand, we have $$k((x,y))^{\delta}=Q(k[[x^3,y^3.\phi^2\psi,\phi\psi^2]])=k((x^3,y^3))(\frac{\phi}{\psi}).$$
Since $k((x,y))^{\delta'} \subseteq k((x,y))^{\delta}$ and both are  field extensions of $k((x^3,y^3))$ of degree $3$, it follows that $k((x,y))^{\delta}=k((x,y))^{\delta'}$.
By Proposition \ref{Jacobson}, there exists $\rho \in k((x,y))$ such that $\delta'= \rho\delta$. 

Write 
\begin{align*}
\delta&=f_1\partial_x+f_2\partial_y,\;\text{where}\;f_1,f_2 \in k[[x,y]],\;\text{gcd}(f_1,f_2)=1,\\
\delta'&=f_1'\partial_x+f_2'\partial_y, \;\text{where}\; f'_1=\phi_y\psi-\phi\psi_y,f'_2=\phi_x\psi-\phi\psi_x,\\
\rho&=\frac{b_1}{b_2}, \;\text{where}\;b_1,b_2 \in k[[x,y]],\;\text{gcd}(b_1,b_2)=1.
\end{align*}
By Lemma \ref{coprime}, we have $\text{gcd}(f'_1,f'_2)=1$. Applying $\delta$ and $\delta'$ to $x$ and $y$, we have
$$b_2f_1'=b_1f_1,\;b_2f_2'=b_1f_2.$$
Hence it follows that $b_1 \divides b_2f_1'$ and $b_1 \divides b_2f_2'$. Since $\text{gcd}(b_1,b_2)=1$, $b_1$ divides both $f_1'$ and $f_2'$. Using $\text{gcd}(f'_1,f'_2)=1$, we have $b_1 \in k[[x,y]]^{\times}$. Similarly, $b_2 \in k[[x,y]]^{\times}$. Thus we obtain $\rho \in k[[x,y]]^{\times}$, which implies that $\delta'$ is a generator of $\cF$.

($\mathrm{i}$)  Suppose that $\phi \in k[[x,y]]^{\times}$ or $\psi \in k[[x,y]]^{\times}$. Without loss of generality, we may assume the former.
Since $\phi^2\psi=\phi^3\cdot\phi^{-1}\psi$ and $\phi\psi^2=\phi^3\cdot(\phi^{-1}\psi)^2$, $\widehat{A}$ is expressed as
$$\widehat{A}=k[[x^3,y^3,t]],$$
where $t=\phi^{-1}\psi$. Thus $\widehat{A}$ has a hypersurface singularity.
By the same argument as in the case of $p=2$, we see that $\widehat{A}$ is a rational double point.

($\mathrm{ii}$) Suppose that $\phi,\psi \in \fm$. Write 
$$\phi=ax+by+P \;\; \mathrm{and} \;\; \psi=cx+dy+Q,$$
with $a,b,c,d \in k$ and $P,Q \in \fm^2$. From these expressions, the coefficients of $\partial_x$ and $\partial_y$ in $\delta'$ satisfy 
\begin{align*}
\phi_y\psi-\phi\psi_y &= (b+P_y)(cx+dy+Q)-(d+Q_y)(ax+by+P) \\
&=-(ad-bc)x+bQ+P_y\psi-dP-Q_y\phi, \\
\phi_x\psi-\phi\psi_x &= (a+P_x)(cx+dy+Q)-(c+Q_x)(ax+by+P) \\
&=(ad-bc)y+aQ+P_x\psi-cP-Q_x\phi.
\end{align*}
Since $P_x,P_y,Q_x,Q_y \in \fm$, we have $\delta' \in \fm^2\Der{k[[x,y]]}$ if and only if $ad-bc=0$. By Proposition \ref{m2}, 
the assumption that $\cF$ is $\frac{2}{3}$-klt implies $ad-bc \neq 0$.
Therefore $\phi,\psi$ form a system of parameters of $k[[x,y]]$. Then there exists an automorphism of $k[[x,y]]$ sending $\phi$ and $\psi$ to $x$ and $y$, respectively. In particular, we have $x^3,y^3 \in k[[\phi^3,\psi^3]]$. Hence $\widehat{A}$ is expressed as
$$\widehat{A}=k[[x^3,y^3,\phi^2\psi,\phi\psi^2]]=k[[\phi^3,\psi^3,\phi^2\psi,\phi\psi^2]].$$
Thus $\widehat{A}$ has a toric singularity of type $\frac{1}{3}(1,1)$.
\end{proof}

\subsection{Case $p=5$}
In the case $p=5$, we determine the klt quotients by 1-foliations that are not lc.

\begin{thm} \label{4/5-klt}
Let $\cF$ be an invertible $\frac{4}{5}$-klt 1-foliation on $S$. If $\cF$ is not lc at $q \in S$, then the singularity at $\tilde{q} \in S/\cF$ is a rational double point of type $E_8^0$. 
\end{thm}

Before the proof, we need the following lemma. 

\begin{lem} \label{expression}
Let $S=\mathrm{Spec}\;k[[x,y]]$ and let $\cF$ be a $\frac{4}{5}$-klt 1-foliation on $S$ generated by $\delta \in \fm\Der{k[[x,y]]}$. If $\cF$ is not lc, then, after a suitable change of coordinates, $\delta$ is written as
$$\delta=y\dd{x}+(x^2+g)\dd{y},$$
with $g \in (x^3,xy,y^2)$.
\end{lem}

\begin{proof}
Since $\cF$ is not lc, $\delta|_{\fm/\fm^2}$ is nilpotent by Proposition \ref{nilpotence}. If $\delta|_{\fm/\fm^2}=0$, then we have $\delta\in \fm^2\Der{k[[x,y]]}$ which contradicts the assumption that $\cF$ is $\frac{4}{5}$-klt by Lemma \ref{t'-klt} and Proposition \ref{m2}. Thus $\delta|_{\fm/\fm^2} \neq 0$. Since $\delta|_{\fm/\fm^2}$ is non-zero but nilpotent, it is represented by the Jordan block
$$\begin{pmatrix}
    0 & 1 \\
    0 & 0
\end{pmatrix}$$
for a suitable basis. Thus, changing coordinates if necessary, we may assume that 
$$\delta=(y+\psi)\dd{x}+\phi\dd{y},$$ with $\phi,\psi \in \fm^2$. 
Also, by the further coordinate change $y'=y+\psi$, we may assume that $\psi=0$. Write $\phi=cx^2+g(x,y)$, with $c \in k$ and $g(x,y) \in (x^3,xy,y^2)$.

Let $\pi_1 \colon S_1 \to S$ be the blow-up at the origin. One chart of the blow-up is given by
$$k[[x,y]] \to k[[x_1,y_1]][t]/(y_1-x_1t), \;\; (x,y) \mapsto (x_1,x_1t).$$
By the transformation rules \eqref{pullback formula}, $\pi_1^*\delta$ is calculated as follows:
\begin{align*}
\pi_1^*\delta&=x_1t\cdot\left(\dd{x_1}-\frac{t}{x_1}\dd{t}\right)+\left(cx_1^2+g(x_1,x_1t)\right)\cdot\frac{1}{x_1}\dd{t} \\
&=x_1t\dd{x_1}+\left(cx_1-t^2+\frac{g(x_1,x_1t)}{x_1}\right)\dd{t}.
\end{align*}
Since $g \in \fm^2$, we see that $g(x_1,x_1t)/x_1$ does not have a $t^2$ term. Hence the coefficient of $\partial_t$ in $\pi_1^*\delta$ is not divisible by $x_1$. Thus the pullback foliation $\pi_1^*\cF$ on this chart is generated by $\pi_1^*\delta$. In particular, we obtain
\begin{equation} \label{canonical div5}
    K_{\pi_1^*\cF}=\pi_1^*K_{\cF}.
\end{equation}

Let $q' \in S_1$ be the closed point corresponding to the maximal ideal $\fm'=(x_1,y_1)$. If $c=0$, then we have $\pi_1^*\delta \in \fm'^2 \Der{\cO_{S_1,q'}}$. Let $\pi_2 \colon S_2 \to S_1$ be the blow-up at $q'$ and $E'$ the unique exceptional divisor. Since $q' \in E$, we have $a(E';S)=a(E;S)+1=2$. Using Lemma \ref{blow-up} and equation \eqref{canonical div5}, we have $a(E';\cF)=a(E';\pi_1^*\cF) \leq -1-\epsilon_{\cF}(E')$. Thus
$$\frac{4}{5}a(E';\cF)+\frac{1}{5}a(E';S) \leq \frac{4}{5}(-1-\epsilon_{\cF}(E'))+\frac{2}{5} = -\frac{4}{5}\epsilon_{\cF}(E')-\frac{2}{5} \leq -\frac{4}{5}\epsilon_{\cF}(E')-\frac{1}{5}.$$
This contradicts the assumption that $\cF$ is $\frac{4}{5}$-klt. Thus we have $c \neq 0$.

By the coordinate change $x'=cx$ and $y'=cy$, we obtain the desired expression for $\delta$.
\end{proof}

Now we prove Theorem \ref{4/5-klt}.
The key idea of the proof in this case is \cite[Lemma 3.6]{Matsu2}.

\begin{proof}[Proof of Theorem \ref{4/5-klt}]
We may assume that $S=\text{Spec}\;k[[x,y]]$ and $\cF$ is generated by $\delta\in\Der{k[[x,y]]}$ written as in Lemma \ref{expression}. Since $\delta$ is not multiplicative, we can write $\delta^5=\alpha\delta$ with $\alpha \in \fm$.

Now we consider $k[[x,y]]$ as a weighted ring with $\deg(x)=2$ and $\deg(y)=3$ and let $I_n$ be the ideal generated by monomials of degree $\geq n$. Then, we see that $g \in I_5$. If $x^iy^j$ is a monomial in $I_n$, then we have
$$\delta(x^iy^j)=ix^{i-1}y^{j+1}+jx^{i+2}y^{j-1}+jx^iy^{j-1}g.$$
Since $\deg(x^{i-1}y^{j+1})=2(i-1)+3(j+1)\geq n+1,\;\deg(x^{i+2}y^{j-1})=2(i+2)+3(j-1)\geq n+1$, and $\deg(x^iy^{j-1}g) \geq 2i+3(j-1)+5\geq n+2$, it follows that
\begin{equation} \label{deltaIn}
    \delta(I_n) \subseteq I_{n+1}.
\end{equation}

Next we will show that $\alpha \in I_5$.
Write $\alpha=a_2x+a_3y+a_4x^2+f$ with $a_2,a_3,a_4 \in k$ and $f \in I_5$.
By Lemma \ref{alpha}, we have
$$\delta(\alpha)=a_2y+a_3x^2+2a_4xy+a_3g+\delta(f)=0.$$
From \eqref{deltaIn}, it follows that $\delta(f) \in I_6$. Hence,  
considering the terms of degree $3,4$ and $5$, we get $a_2=a_3=a_4=0$. Thus we have $\alpha \in I_5$.

We compute $\delta^4(x)$ as follows: 
$$\delta^4(x)=\delta^2(\phi)=\delta(2xy+\delta(g))=2y^2+2x^3+2xg+\delta^2(g).$$
Since $g,\alpha \in I_5$, we have $\deg(2xg)=\deg(g)+2 \geq 7$ and $\deg(\alpha x)=\deg(\alpha)+2 \geq 7$. Also, using \eqref{deltaIn}, we have $\delta^2(g) \in I_7$. Thus we obtain
\begin{equation} \label{h}
    h\coloneqq\delta^4(x)-2(y^2+x^3)-\alpha x=2xg+\delta^2(g)-\alpha x \in I_7.
\end{equation}

Let $\tau=\delta^4(x)-\alpha x$. Then we have
\begin{equation} \label{tauh}
\tau=2(y^2+x^3)+h.
\end{equation}
From Lemma \ref{alpha}, it follows that
$$\delta(\tau)=\delta(\delta^4(x)-\alpha x)=\delta^5(x)-\alpha \delta(x) -\delta(\alpha)x=0.$$
Hence we have $\tau \in k[[x,y]]^{\delta}$. We consider the homomorphism
$$\theta:k[[u,v,t]] \to k[[x,y]]^{\delta};\;\;(u,v,t) \mapsto (x^5,y^5,\tau).$$
Using equation \eqref{(p)} and \eqref{tauh}, we obtain
$$\theta(t^5-2(v^2+u^3)-h^{(5)})=\tau^5-2x^{10}-2y^{15}-h^5=\tau^5-(2(y^2+x^3)+h)^5=0.$$
Thus it follows that $t^5-2(v^2+u^3)-h^{(5)} \in \text{Ker}\;\theta$.
Moreover, since $I_7=(x^4,x^2y,xy^2,y^3)$, \eqref{h} implies 
\begin{equation} \label{h5}
    h^{(5)} \in (u^4,u^2v,uv^2,v^3)_{k[[u,v]]}.
\end{equation}
Thus $t^5-2(v^2+u^3)-h^{(5)}$ is irreducible.

Suppose that $\text{Ker}\;\theta \neq (t^5-2(u^3+v^2)-h^{(5)})$. Then it follows that
\begin{equation} \label{dim}
    \text{dim}\;k[[u,v,t]]/\text{Ker}\;\theta \leq 1.
\end{equation}
On the other hand, we have natural inclusions
$$k[[u,v]] \subseteq k[[u,v,t]]/\text{Ker}\;\theta \subseteq k[[x,y]]^{\delta}.$$
Therefore \eqref{dim} contradicts the fact that $\text{Spec}\;k[[u,v]]$ is homeomorphic to $\text{Spec}\;k[[x,y]]^{\delta}$.
From the above, we obtain the injective homomorphism
$$\bar{\theta}:B\coloneqq k[[u,v,t]]/(t^5-2(u^3+v^2)-h^{(5)}) \to k[[x,y]]^{\delta}$$
induced by $\theta$. In what follows, we regard $B$ as the subring of $k[[x,y]]^{\delta}$ via $\bar{\theta}$. 

Let $\mathfrak{n}=(u,v)$ be the maximal ideal of $k[[u,v]]$. Using \eqref{h5}, we can write $$h^{(5)}=c'u^2v+u^3g_1(u,v)+v^2g_2(u,v),$$
with $c'\in k$ and $g_1(u,v),g_2(u,v) \in \mathfrak{n}$.
By the coordinate change $v'=v+c'u^2$, we can calculate as follows:
$$2v^2+c'u^2v=2(v'-c'u^2)^2-c'u^2(v'-c'u^2)=2v'^2-4c'u^2v'-c'u^2v'+c'^2u^4=2v'^2+c'^2u^4.$$
Thus the defining polynomial of $B$ is given by
$$t^5-2(v'^2+u^3)+u^3g'_1(u,v')+v'^2g'_2(u,v'),$$
with $g'_1(u,v'),g'_2(u,v') \in \mathfrak{n}$.
Moreover, by the coordinate change
$$X=v'(-2+g'_2(u,v'))^{1/2}, Y=u(-2+g'_1(u,v'))^{1/3}, Z=t,$$
we obtain
$$B \cong k[[X,Y,Z]]/(X^2+Y^3+Z^5).$$
Since $B$ is normal and $Q(B)=Q(k[[x,y]]^{\delta})$, we have $B=k[[x,y]]^\delta$ and it is a rational double point of type $E_8^0$.
\end{proof}

Combining Theorem \ref{4/5-klt} with Theorem \ref{multi}, we obtain the following corollary.

\begin{cor} \label{p=5}
If $\cF$ is $\frac{4}{5}$-klt 1-foliation on $S$, then each singular point of $S/\cF$ is either
\begin{enumerate}
\item a rational double point of type $A_4$ or $E_8^0$, or
\item a toric singularity of type $\frac{1}{5}(1,1)$ or $\frac{1}{5}(1,2)$.
\end{enumerate}
\end{cor}
\begin{proof}
We may assume that $S=\mathrm{Spec}\;k[[x,y]]$ and $\cF$ is generated by $\delta \in \Der{k[[x,y]]}$. If $\delta$ is multiplicative, then $k[[x,y]]^{\delta}$ is a toric singularity of type $\frac{1}{5}(1,\lambda)$ for some $\lambda\in\mathbb{F}_5^\times$, by Theorem \ref{multi}. Here, a toric singularity of type $\frac{1}{5}(1,4)$ given by
$$k[[x^5,y^5,xy,x^2y^2,x^3y^3,x^4y^4]]=k[[x^5,y^5,xy]] \cong k[[X,Y,Z]]/(Z^5-XY),$$
corresponds to a rational double point of type $A_4$. Also, by permuting the coordinates and multiplying the weights by $2$ modulo $5$, we have $\frac{1}{5}(1,3) \cong \frac{1}{5}(3,1) \cong \frac{1}{5}(1,2)$.

If $\delta$ is not multiplicative, then $k[[x,y]]^{\delta}$ is a rational double point of type $E_8^0$ by Theorem \ref{4/5-klt}. Therefore the classification is complete.
\end{proof}

\bibliographystyle{alpha}
\bibliography{reference}

\end{document}